\documentclass[preprint,12pt]{elsarticle}

\usepackage{amsmath,amssymb,amsfonts,amsthm,amscd}
\newtheorem{thm}{Theorem}
\newtheorem{lem}[thm]{Lemma}
\newdefinition{rmk}{Remark}
\newproof{pf}{Proof}
\newdefinition{proposition}{Proposition}
\newdefinition{assumption}{Assumption}
\def\div{\mathop{\mathrm{div}}\nolimits}
\def\!{\mathop{\mathrm{!}}}

\def\R{\mathbf{ R}}

\def\X{\tilde {X}}

\def\P{\tilde{P}}
\def\Q{\tilde{Q}}
\def\Z{\tilde{Z}}

\def\Qq{\mathcal{Q}}

\def\E{\mathbb{E}}
\def\D{\mathcal{D}}

\def\M{\mathcal{M}}

\journal{***}

\begin{document}

\begin{frontmatter}



\title{Long time behaviour and particle approximation of a generalized Vlasov dynamic}


\author{Manh Hong Duong}

\address{Mathematics Institute,\\
University of Warwick, \\
Coventry CV4 7AL, UK. \\
Email: m.h.duong@warwick.ac.uk}

\begin{abstract}
In this paper, we are interested in a generalised Vlasov equation, which describes the evolution of the probability density of a particle evolving according to a generalised Vlasov dynamic. The achievement of the paper is twofold. Firstly, we 
obtain a quantitative rate of convergence to the stationary solution in the Wasserstein metric. Secondly, we provide a many-particle approximation for the equation and show  that the approximate system satisfies the propagation of chaos property.
\end{abstract}

\end{frontmatter}


\section{Introduction}
\subsection{The main equation}
This paper is concerned with the long time behaviour and  particle approximation of solutions of the following equation
\begin{equation}
\label{eq: GLE}
\partial_t\rho_t=-p\cdot\nabla_q\rho_t+\left[\beta q+ A(q)+B\ast \rho_t(q)-\lambda z\right]\cdot \nabla_p\rho_t+\div_z[(\lambda p+\alpha\, z)\rho_t]+\Delta_z\rho_t.
\end{equation}
The spatial domain is $\R^{3d}$ with coordinates $(q,p,z)$, with $q,p$ and $z$ each in $\R^d$. Subscripts as in $\nabla_q$ and $\Delta_z$ are used to indicate that the differential operators act only on those variables. The unknown is a time-dependent probability measure $\rho\colon[0,T]\to \mathcal{P}(\R^{3d})$; $A$ and $B$ are given $\R^d$-to-$\R^d$ maps; $\alpha,\beta$ and $\lambda$ are given positive constants. Finally, the convolution $B\ast \rho_t(q)$ is defined by 
\begin{equation}
\label{eq: convolution}
B\ast \rho_t(q)=\int_{\R^{3d}} B(q-q')\rho_t(q',p,z) dq'dpdz.
\end{equation}
It is well-known that Eq. \eqref{eq: GLE} is the forward Kolmogorov equation of the following stochastic differential equation (SDE)
\begin{equation}
\label{eq: SDE}
\begin{cases}
dQ_t=P_t\,dt,
\\dP_t=- \beta\, Q_t\,dt-A(Q_t)\,dt-B\ast\rho_t(Q_t)\,dt+\lambda\,Z_t\,dt,
\\dZ_t=-\lambda\,P_t\,dt-\alpha\, Z_t\,dt+\sqrt{2}\,dW_t,
\end{cases}
\end{equation}
where $(W_t)_{t\geq 0}$ is the $d$-dimensional standard Wiener process.

The law of the $\R^{3d}$-valued process $(Q_t,P_t,Z_t)_{t\geq 0}$ evolving according to the SDE \eqref{eq: SDE} is a solution of Eq. \eqref{eq: GLE} at time $t$. Systems of the form as in \eqref{eq: GLE} and \eqref{eq: SDE} have been studied by many authors. When $B=0$ and $A$ is a gradient of a potential $A=\nabla V$, \eqref{eq: SDE} is called a generalized Langevin dynamic in \cite{Kup04,OP11}. It is also a special case of the so-called Nose-Hoover-Langevin dynamic studied in \cite{SDC07,LNT09,BFL11,FG11,BFL13}. In these papers \eqref{eq: SDE} is considered as a stochastic perturbation, via the external heat bath variable $Z_t$, of the classical deterministic Hamiltonian system for $(Q_t,P_t)$. The generalised Langevin and the Nose-Hoove-Langevin dynamics play an important role in molecular dynamics because they maintain canonical sampling of the original Hamiltonian system while being ergodic. When $B\neq 0$, the involution term $B\ast\rho_t$, which is often known as Vlasov term or mean-field term in the literature, models the self-interaction of the underlying physical system. In this case, \eqref{eq: SDE} belongs to the class of weakly interacting particle systems and have been used widely in applied sciences, especially in plasma physics, see e.g., \cite{DG87,DG89,Mel96,DPZ13b} and references therein. Combining both cases, the full dynamic \eqref{eq: SDE} can also be seen as stochastic perturbation (via the augmented variable $Z$) of Vlasov's dynamic similarly as the generalised Langevin dynamic. This motivates the name in the title of the present paper. 

The purpose of this paper is twofold. First, we prove the existence and uniqueness of a stationary solution as well as convergence to the stationary solution of other solutions of \eqref{eq: GLE}. Second, we provide a particle approximation of solutions of \eqref{eq: GLE}. We address both issues using the stochastic interpretation \eqref{eq: SDE} of \eqref{eq: GLE} and the so-called coupling technique introduced in \cite{Tal02,Vil09} and used further in \cite{BGM10,BCC11}. We now state the assumptions and main results of the present paper.

\subsection{Assumptions and main results of the paper}
Throughout the paper, we make the following assumption.
\begin{assumption}
\label{ass: assumption 1}
$B$ is an odd map on $\R^{3d}$. Furthermore, $A$ and $B$ are Lipschitz, i.e., there are non-negative constants $C_A, C_B$ such that for all $q_1,q_2\in \R^d$
\begin{equation}
\label{Lipschitz}
|A(q_1)-A(q_2)|\leq C_A |q_1-q_2|, \quad |B(q_1)-B(q_2)|\leq C_B|q_1-q_2|.
\end{equation}
\end{assumption}
The assumption that $B$ is odd is motivated by the fact that in applications, $B$ is often a gradient of a symmetric interaction potential which depends only on the distance of two different particles in the underlying physical system. The Lipschitz conditions \eqref{Lipschitz} are standard assumptions to ensure the existence and uniqueness of the SDE \eqref{eq: SDE} with square-integrable initial data, see e.g., \cite{Mel96}. It also guarantees the well-posedness of Eq. \eqref{eq: GLE}: it admits a unique measure solution $\rho_t\in \mathcal{P}_2(\R^{3d})$ for any $t>0$ provided that the initial data $\rho_0\in \mathcal{P}_2(\R^{3d})$, where $\mathcal{P}_2(\R^{3d})$ denotes the subspace of probability measures on $\R^{3d}$ with finite second moments.

The first result of the present paper is the following theorem concerning the existence and uniqueness of a stationary solution as well as convergence towards to the stationary solution of all solutions of \eqref{eq: GLE}. The trend to equilibrium is measured in terms of the Wasserstein distance $W_2$ on $\mathcal{P}_2(\R^{3d})$. In particular, $W_2(\mu,\nu)$, with $\mu,\nu\in\mathcal{P}_2(\R^{3d})$, can be formulated as the \textit{infimum} over a set of all coupled random variables that distribute according to $\mu$ and $\nu$ respectively (more details follow in Definition \ref{eq: def W_2}). The usage of the Wasserstein distance has two advantages. Firstly, it makes use of the relationship between the SDE \eqref{eq: SDE} and Eq. \eqref{eq: GLE}, which is the underlying idea of the coupling technique, see Section \ref{sec: preliminary}. Secondly, the infimum formulation of the Wasserstein distance is convenient in estimating upper bounds, which is needed to prove convergence, by choosing any admissible couple of random variables.

\begin{thm}
\label{theo: longtime}
Under Assumption \ref{ass: assumption 1}, for all positive $\alpha,\beta$ and $\lambda$ there exists a positive constant $\eta_0$ such that, if $0\leq C_A+C_B<\eta_0$, then there exist positive constant $C$ and $C'$ such that
\begin{equation}
\label{eq: assertion 1 thm1}
W_2(\rho_t,\overline{\rho}_t)\leq C' \, e^{-C \, t}\, W_2(\rho_0,\overline{\rho}_0), \quad\text{for all}\quad t\geq 0,
\end{equation}
for all solutions $(\rho_t)_{t\geq 0}$ and $(\overline{\rho}_t)_{t\geq 0}$ of \eqref{eq: GLE} with respectively initial data $\rho_0$ and $\overline{\rho}_0$ in $\mathcal{P}_2(\R^{3d})$. Moreover, \eqref{eq: GLE} has a unique stationary solution $\rho_{\infty}$ and all solutions $(\rho_t)_{t\geq 0}$ converge towards it exponentially in the Wasserstein distance
\begin{equation}
\label{eq: assertion 2 thm1}
W_2(\rho_t,\rho_\infty)\leq C'\,e^{-C\, t}\,W_2(\rho_0,\rho_\infty).
\end{equation}
\end{thm}
We recall that $C_A$ and $C_B$ are respectively the Lipschitz constants of the confinement $A$ and interaction $B$ forces. The condition that $0\leq C_A+C_B<\eta_0$ in the theorem means that they are linear-like forces, and that the interaction is weak enough. The proof of this theorem is given in Section \ref{sec: Long time}.

We now describe the setup and state the second main result of the present paper. Let $(W^i_\cdot)_{i\geq 1}$ be independent standard Brownian motions on $\R^d$, and $(Q^{i}_0,P^{i}_0,Z^{i}_0)_{i\geq 1}$ be independent random vectors on $\R^{3d}$ with law $\rho_0\in \mathcal{P}(\R^{3d})$ and independent of $(W^i_\cdot)_{i\geq 1}$. Let $(Q^{(N)}_t,P^{(N)}_t,Z^{(N)}_t)_{t\geq 0}=(Q^{1,N}_t,\ldots,Q^{N,N}_t,P^{1,N}_t,\ldots,P^{N,N}_t,Z^{1,N}_t,\ldots,Z^{N,N}_t)_{t\geq 0}$ be the solution of the following SDE in $(\R^{3d})^N$:
\begin{equation}
\label{eq: many-SDE}
\begin{cases}
dQ^{i,N}_t=P^{i,N}_t\,dt,
\\dP^{i,N}_t=- \beta\, Q^{i,N}_t\,dt-A(Q^{i,N}_t)\,dt-\frac{1}{N}\sum_{j=1}^N B(Q^{i,N}_t-Q^{j,N}_t)\,dt+\lambda\,Z^{i,N}_t\,dt,
\\dZ^{i,N}_t=-\lambda\,P^{i,N}_t\,dt-\alpha\, Z^{i,N}_t\,dt+\sqrt{2}\,dW^{i}_t,
\\(Q^{i,N}_0,P^{i,N}_0,Z^{i,N}_0)=(Q^{i}_0,P^{i}_0,Z^{i}_0).
\end{cases}
\end{equation}
This is an interacting particle system since there is an interacting term $\frac{1}{N}\sum_{j=1}^N B(Q^{i,N}_t-Q^{j,N}_t)$ in the SDE for each $i$. However, the interaction is of weak type in the sense that the interacting term for each $i$ depends only on the empirical measure $\mu^N_t$ associated to the whole system
\begin{equation*}
\frac{1}{N}\sum_{j=1}^N B(Q^{i,N}_t-Q^{j,N}_t)=\int_{\R^{3d}} B(Q^{i,N}_t-q)\mu^N_t(dqdpdz),
\end{equation*}
where
\begin{equation}
\label{eq: empirical}
\mu^N_t(dqdpdz):=\frac{1}{N}\sum_{i=1}^N \delta_{(Q^{i,N}_t,P^{i,N}_t,Z^{i,N}_t)}(dqdpdz).
\end{equation}

On the other hand, we consider the following non-interacting particle system, which is a many-particle version of the SDE \eqref{eq: SDE}
\begin{equation}
\label{eq: many-SDE 2}
\begin{cases}
d\Q^{i}_t=\P^{i}_t\,dt,
\\d\P^{i}_t=- \beta\, \Q^{i}_t\,dt-A(\Q^{i}_t)\,dt-B\ast\mu_t(\Q^i_t)\,dt+\lambda\,\Z^{i}_t\,dt,
\\d\Z^{i}_t=-\lambda\,\P^{i}_t\,dt-\alpha\, \Z^{i}_t\,dt+\sqrt{2}\,dW^{i}_t,
\\(\Q^{i}_0,\P^{i}_0,\Z^{i}_0)=(Q^{i}_0,P^{i}_0,Z^{i}_0),
\end{cases}
\end{equation}
where $\mu_t$ is the distribution of $(\Q^{i}_t,\P^{i}_t,\Z^{i}_t)$. Since the initial data and the driving Brownian motions are independent $(Q^i_t,P^i_t,Z^i_t)_{t\geq}$ with $i\geq 1$ are independent. Furthermore, they are identically distributed and their common law evolves according to \eqref{eq: GLE}, so $\mu_t$ is the solution of \eqref{eq: GLE} at time $t$ with initial data $\rho_0$. In comparison with  the system in \eqref{eq: many-SDE}, the interacting term is replaced by $B\ast \mu_t(\Q^i_t)$, i.e., integrating of $B(Q^{i,N}_t-\cdot)$ with respect to $\mu_t$ instead of the empirical measure $\mu^N_t$. 

In the second main result of the present paper, we prove that as $N$ becomes larger and larger, the weakly interacting processes $(Q^{i,N}_t,P^{i,N}_t,Z^{i,N}_t)_{t\geq 0}$ in \eqref{eq: many-SDE} behaves more and more like the independent processes $(\Q^{i}_t,\P^{i}_t,\Z^{i}_t)_{t\geq 0}$ in \eqref{eq: many-SDE 2}. In other words, the particle system \eqref{eq: many-SDE} satisfies the property of propagation of chaos as defined in \cite{Szn91}, which we recall in Section \ref{sec: preliminary}.
\begin{thm}
\label{theo: progation of chaos}
Let $(Q^{i}_0,P^{i}_0,Z^{i}_0),~ i=1,\ldots, N$ be $N$ independent $\R^{3d}$-valued random variables with law $\rho_0\in \mathcal{P}_2(\R^{3d})$. Let $(Q^{i,N}_t,P^{i,N}_t,Z^{i,N}_t)_{t\geq 0, 1\leq i\leq N}$ and $(\Q^{i}_t,\P^{i}_t,\Z^{i}_t)_{t\geq 0}$  respectively be the solution of \eqref{eq: many-SDE} and \eqref{eq: many-SDE 2} with initial datum $(Q^{i}_0,P^{i}_0,Z^{i}_0),~ i=1,\ldots, N$. Under Assumption \ref{ass: assumption 1}, for all positive $\alpha,\beta,\lambda$ there exists a positive constant $\eta_0$ such that, if $0\leq C_A+C_B<\eta_0$, then there exists a positive constant $C$, independent of $N$, such that for all $i=1,\ldots, N$
\begin{equation}
\label{eq: propagation of chaos}
\sup_{t\geq 0} \E\left(\left|Q^{i,N}_t-\Q^i_t\right|^2+\left|P^{i,N}_t-\P^i_t\right|^2+\left|Z^{i,N}_t-\Z^i_t\right|^2\right)\leq \frac{C}{N}.
\end{equation}
As a consequence, the particle system \eqref{eq: many-SDE} satisfies the property of propagation of chaos. In addition, the law $\rho^{(1,N)}_t$ of any $(Q^{i,N}_t,P^{i,N}_t,Z^{i,N}_t)$ converges to the law $\mu_t$ of $(\Q^{i}_t,\P^{i}_t,\Z^{i}_t)$ uniformly in $t$ and in the Wasserstein metric
\begin{equation}
\label{eq: Wasserstein bound}
\sup_{t\geq 0} W_2(\rho^{(1,N)}_t, \rho_t)^2\leq \frac{C}{N}.
\end{equation}
\end{thm}
We stress that this theorem not only proves the convergence but also provides a quantitative rate of convergence. This theorem is proved in Section \ref{sec: propa}.
\subsection{Organization of the paper}
The rest of the paper is organized as follows. In Section \ref{sec: preliminary}, we recall some preliminary knowledge. We present the proof of Theorem \ref{theo: longtime} in Section \ref{sec: Long time} and that of Theorem \ref{theo: progation of chaos} in Section \ref{sec: propa}.
\section{Preliminaries}
\label{sec: preliminary}
\subsection{Wasserstein metric}
We recall that $\mathcal{P}_2(\R^{3d})$ denotes the space of probability measures $\mu$ on $\R^{3d}$ with finite second moment, i.e.,
\begin{equation*}
\int_{\R^{3d}}[|q|^2+|p|^2+|z|^2]\,\mu(dqdpdz)<\infty.
\end{equation*}
The Wasserstein distance $W_2$ between two probability measures $\mu,\nu\in \mathcal{P}_2(\R^{3d})$ is defined via
\begin{equation}
\label{eq: def W_2}
W_2(\mu,\nu)^2 = \inf_{(Q,P,Z),(\Q,\P,\Z)} \E \Big(|Q-\Q|^2+|P-\P|^2+|Z-\Z|^2\Big),
\end{equation}
where the infimum are taken over all triples $(Q,P,Z)$ and $(\Q,\P,\Z)$ of random variables on $\R^{3d}$ with laws $\mu$ and $\nu$ respectively. We note that convergence in Wasserstein metric is equivalent to narrow convergence (i.e., tested against continuous and bounded functions) plus convergence of the second moments. The Wassertein metric plays an important role in many fields of mathematics such as optimal transport and dissipative evolution equations. We refer to \cite{Vil03} and \cite{AGS08} for expositions on the topic. 

More generally, given a positive definite quadratic form $\Qq$ on $\R^{3d}$, one can define a $\Qq$-Wasserstein distance $W_{2\Qq}(\mu,\nu)$ between $\mu,\nu\in \mathcal{P}_2(\R^{3d})$ via
\begin{equation}
\label{eq: def W_2Q}
W_{2\Qq}(\mu,\nu)^2 = \inf_{(Q,P,Z),(\Q,\P,\Z)} \E \Qq \Big(Q-\Q,P-\P,Z-\Z\Big),
\end{equation}
where the infimum is also taken over all triples $(Q,P,Z)$ and $(\Q,\P,\Z)$ of random variables on $\R^{3d}$ with laws $\mu$ and $\nu$ respectively. Moreover, since $\sqrt{\Qq}$ is equivalent to the Euclidean norm $\sqrt{|q|^2+|p|^2+|z|^2}$, there exist positive constants $C$ and $C'$ such that 
\begin{equation}
\label{eq: equivalent of W2 and W2Q}
C W_2(\mu,\nu)\leq W_{2\Qq}(\mu,\nu)\leq C' W_2(\mu,\nu),
\end{equation}
for $\mu,\nu\in \mathcal{P}_2(\R^{3d})$.

The underlying idea of the coupling technique introduced in \cite{Tal02,Vil09} and used further in \cite{BGM10,BCC11} is based on \eqref{eq: equivalent of W2 and W2Q}. First, we find a positive definite quadratic form $\Qq$ on $\R^{3d}$ such that \eqref{eq: SDE} is dissipative with respect to $\Qq$. Then we apply \eqref{eq: equivalent of W2 and W2Q} to obtain estimations in terms of the Wasserstein metric.
We note that the Wasserstein distance can be defined via a variety of different ways, but for the purpose of this paper, the definition above will be the most useful because of two reasons as already mentioned in the introduction: it is facilitated with the coupling technique and is convenient in obtaining upper bound estimates. 
\subsection{Propagation of chaos}
We now recall the definition of propagation of chaos of a particle system introduced in \cite{Szn91}. Letting $X^{i,N}_t:=(Q^{i,N}_t,P^{i,N}_t,Z^{i,N}_t)$, and denoting $dx^{i,N}=dq^{i,N}dp^{i,N}dz^{i,N}$ for $i=1,\ldots,N$. Let $\rho^{(N)}_t(dx^{1,N},\ldots, dx^{N,N})$ be the join law of $(X^{1,N},\ldots,X^{N,N})$ where for each $i=1,\ldots, N$, $X^{i,N}_t$ satisfies \eqref{eq: many-SDE}. Let $k\geq 1$ be a fixed integer, and let $(i_1,\ldots,i_k)$ be any $k$-uple of $[1,N]$. The $(i_1,\ldots,i_k)$-marginal $\rho^{(i_1,\ldots,i_k)}$ of $\rho^{(N)}_t$ is defined via
\begin{equation*}
\rho^{(i_1,\ldots,i_k)}(A):= \rho^{(N)}_t(A\times \R^{3d(N-k)}),
\end{equation*}
for any Borel set $A\subset \R^{(3d)k}$.
We recall that $\mu_t$ is the law of $(\Q^1_t,\P^1_t,\Z^1_t)$, which solves Eq. \eqref{eq: GLE}. The particle system \eqref{eq: many-SDE} is said to  satisfy the property of propagation of chaos with respect to $\mu_t$ if $\rho^{(i_1,\ldots,i_k)}$ converges narrowly to $(\mu_t)^{\otimes k}$ for all $k$-uple $(i_1,\ldots,i_k)$.
According to \cite[Proposition 2.2]{Szn91}, this is equivalent to the statement that the empirical measure $\mu^N_t$ converges narrowly to $\mu_t$. We refer to \cite{Szn91} for more information on the topic.
\section{Proofs of the main theorems}
The proofs of Theorem \ref{theo: longtime} and Theorem \ref{theo: progation of chaos} are presented respectively in Section \ref{sec: Long time} and Section \ref{sec: propa}. The proofs follow the procedure of \cite{BGM10,BCC11} (see also \cite{Tal02,Vil09}) which use the stochastic interpretation \eqref{eq: SDE} of \eqref{eq: GLE} and the coupling technique explained in the previous section.
\subsection{Stationary solution and long time behaviour of the main equation: proof of Theorem \ref{theo: longtime}}
\label{sec: Long time}
The main ingredient of the proof of Theorem \ref{theo: longtime} is the following proposition, which is the content of the coupling technique.
\begin{proposition}
\label{pro: quadratic}
Under the assumption of Theorem \ref{theo: longtime}, there exists a positive constant $C$ and a positive definite quadratic form $\Qq$ on $\R^{3d}$ such that
\begin{equation}
\label{eq: quadratic}
W_{2\Qq}(\rho_t,\overline{\rho}_t)\leq e^{-C\,t}W_{2\Qq}(\rho_0,\overline{\rho}_0),\qquad\text{for all}~~ t\geq 0.
\end{equation}
\end{proposition}
\begin{proof}
Let $(Q_t,V_t,Z_t)$ and $(\overline{Q}_t,\overline{P}_t,\overline{Z}_t)$ be two $\R^{3d}$-valued stochastic processes evolving according to \eqref{eq: SDE} with the same Brownian motion $(W_t)_{t\geq 0}$ in $\R^d$. Let $(q_t,p_t,z_t):=(Q_t-\overline{Q}_t,P_t-\overline{P}_t,Z_t-\overline{Z}_t)$ be the difference between them. Then $(q_t,p_t,z_t)$ evolves according to
\begin{equation}
\label{eq: SDE2}
\begin{aligned}
dq_t&=p_t\,dt,
\\dp_t&=-\beta\, q_t\,dt-(A(Q_t)-A(\overline{Q}_t)\,dt-(B\ast\rho_t(Q_t)-B\ast\overline{\rho}_t(\overline{Q}_t))\,dt+\lambda\,z_t\,dt,
\\dz_t&=-\lambda\,p_t\,dt-\alpha\, z_t\,dt.
\end{aligned}
\end{equation}
Let $a_1,\cdots,a_5$ be positive constants, which will be specified later on. Define
\begin{equation}
\label{eq: Q1}
\Qq (q,p,z)=a_1|q|^2+a_2|p|^2+a_3|z|^2+2a_4 q\cdot p+2a_5 q\cdot z+2 p\cdot z.
\end{equation}
Assume that $(q_t,p_t,z_t)$ evolves according to \eqref{eq: SDE2}. Then we have
\begin{align}
\label{eq: estimate 1}
&\frac{d}{dt} \Qq(q_t,p_t,z_t)\nonumber
\\&\qquad=\frac{d}{dt}(a_1|q_t|^2+a_2|p_t|^2+a_3|z_t|^2+2a_4 q_t\cdot p_t+2a_5 q_t\cdot z_t+2 p_t\cdot z_t)\nonumber
\\&\qquad= 2a_1 q_t\cdot p_t+2 a_2 p_t\cdot[- \beta\, q_t-(A(Q_t)-A(\overline{Q}_t)-(B\ast\rho_t(Q_t)-B\ast\overline{\rho}_t(\overline{Q}_t))+\lambda\,z_t]\nonumber
\\&\qquad\qquad+ 2a_3z_t\cdot[-\lambda\,p_t-
\alpha z_t]+2a_4|p_t|^2\nonumber
\\&\qquad\qquad+2a_4 q_t\cdot[-\beta\, q_t-(A(Q_t)-A(\overline{Q}_t)-(B\ast\rho_t(Q_t)-B\ast\overline{\rho}_t(\overline{Q}_t))+\lambda\,z_t]\nonumber
\\&\qquad\qquad+2a_5 p_t\cdot z_t+2a_5 q_t\cdot[-\lambda\,p_t-
\alpha z_t]+2p_t\cdot[-\lambda\,p_t-
\alpha z_t]\nonumber
\\&\qquad\qquad +2z_t\cdot[- \beta\, q_t-(A(Q_t)-A(\overline{Q}_t)-(B\ast\rho_t(Q_t)-B\ast\overline{\rho}_t(\overline{Q}_t))+\lambda\,z_t]\nonumber
\\&\qquad=(2a_1-2\lambda a_5-2a_2\beta)q_t\cdot p_t+(2\lambda a_4-2\alpha a_5-2\beta)q_t\cdot z_t\nonumber
\\&\qquad\qquad+(2\lambda a_2-2\lambda a_3+2a_5-2\alpha)\,p_t\cdot z_t -2\beta a_4\,|q_t|^2-(2\lambda-2a_4)|p_t|^2\nonumber
\\&\qquad\qquad-(2\alpha a_3-2\lambda)|z_t|^2
 -(2 a_2 p_t+2a_4 q_t+2z_t)(A(Q_t)-A(\overline{Q}_t))\nonumber
\\&\qquad\qquad -(2 a_2 p_t+2a_4 q_t+2z_t)(B\ast\rho_t(Q_t)-B\ast\overline{\rho}_t(\overline{Q}_t)).
\end{align}
Now we estimate the last two terms on the right-hand side of \eqref{eq: estimate 1} using the assumptions on $A$ and $B$. The first term is bounded from above by
\begin{align}
\label{eq: estimate 2}
& -(2 a_2 p_t+2a_4 q_t+2z_t)(A(Q_t)-A(\overline{Q}_t))\nonumber
 \\&\qquad\leq 2 a_2 C_A|q_t||p_t|+2a_4 C_A|q_t|^2+2 C_A|q_t||z_t|\nonumber
 \\&\qquad\leq a_2 C_A (|q_t|^2+|p_t|^2)+2a_4 C_A|q_t|^2+C_A(|q_t|^2+|z_t|^2)\nonumber
 \\&\qquad= (a_2+2a_4+1)C_A |q_t|^2+ a_2 C_A |p_t|^2+C_A |z_t|^2.
\end{align}
The second term is a bit more intricate. Let $\Pi_t$ be the joint law of $(Q_t,P_t,Z_t;\overline{Q}_t,\overline{P}_t,\overline{Z}_t)$ on $\R^{3d}\times \R^{3d}$. Then its marginals on $\R^{3d}$ are distributions $\rho_t$ and $\overline{\rho}_t$ of $(Q_t,P_t,Z_t)$ and $(\overline{Q}_t,\overline{P}_t,\overline{Z}_t)$ respectively.
We have
\begin{align*}
B\ast\rho_t(Q_t)-B\ast\overline{\rho}_t(\overline{Q}_t)&=\int_{\R^{3d}} B(Q_t-q)\,d\rho_t(q,p,z)-\int_{\R^{3d}} B(\overline{Q}_t-\overline{q})\,d\overline{\rho}_t(\overline{q},\overline{p},\overline{z})
\\&=\int_{\R^{6d}}( B(Q_t-q)-B(\overline{Q}_t-\overline{q}))\,d\Pi_t(q,p,z;\overline{q},\overline{p},\overline{z}).
\end{align*}
Hence
\begin{align*}
&-2 \E \left[q_t\cdot (B\ast\rho_t(Q_t)-B\ast\overline{\rho}_t(\overline{Q}_t))\right]
\\&=-2\int_{\R^{12 d}}(Q-\overline{Q})\cdot (B(Q-q)-B(\overline{Q}-\overline{q}))\, d\Pi_t(Q,P,Z;\overline{Q},\overline{P},\overline{Z})\,\,d\Pi_t(q,p,z;\overline{q},\overline{p},\overline{z})
\\&\overset{(*)}{=}-\int_{\R^{12 d}}\left((Q-q)-(\overline{Q}-\overline{q})\right)\cdot (B(Q-q)-B(\overline{Q}-\overline{q}))\, d\Pi_t(Q,P,Z;\overline{Q},\overline{P},\overline{Z})\,\,d\Pi_t(q,p,z;\overline{q},\overline{p},\overline{z})
\\&\leq C_B\int_{\R^{12 d}}\left|(Q-q)-(\overline{Q}-\overline{q})\right|^2\, d\Pi_t(Q,P,Z;\overline{Q},\overline{P},\overline{Z})\,\,d\Pi_t(q,p,z;\overline{q},\overline{p},\overline{z})
\\&=2 C_B \left[\int_{\R^{6 d}}\left|Q-\overline{Q}\right|^2\, d\Pi_t(Q,P,Z;\overline{Q},\overline{P},\overline{Z})-\left(\int_{\R^{6 d}}(Q-\overline{Q})\, d\Pi_t(Q,P,Z;\overline{Q},\overline{P},\overline{Z})\right)^2\right]
\\&\leq 2C_B \E |q_t|^2,
\end{align*}
where we have used the fact that $B$ is odd to obtain the equality $(*)$.  So, we have just proved that
\begin{equation}
\label{eq: estimate 3}
-\E \left[q_t\cdot (B\ast\rho_t(Q_t)-B\ast\overline{\rho}_t(\overline{Q}_t))\right]\leq C_B \E |q_t|^2.
\end{equation}
Similarly, we have
\begin{align*}
&-2 \E \left[p_t\cdot (B\ast\rho_t(Q_t)-B\ast\overline{\rho}_t(\overline{Q}_t))\right]
\\&=-2\int_{\R^{12 d}}(P-\overline{P})\cdot (B(Q-q)-B(\overline{Q}-\overline{q}))\,d\Pi_t(q,p,z;\overline{q},\overline{p},\overline{z})\, d\Pi_t(Q,P,Z;\overline{Q},\overline{P},\overline{Z})
\\&=-\int_{\R^{12 d}}\left((P-p)-(\overline{P}-\overline{p})\right)\cdot (B(Q-q)-B(\overline{Q}-\overline{q}))\,d\Pi_t(q,p,z;\overline{q},\overline{p},\overline{z})\, d\Pi_t(Q,P,Z;\overline{Q},\overline{P},\overline{Z})
\\&\leq C_B\int_{\R^{12 d}}\left|(P-p)-(\overline{P}-\overline{p})\right|\cdot \left|(Q-q)-(\overline{Q}-\overline{q})\right|\,d\Pi_t(q,p,z;\overline{q},\overline{p},\overline{z})\, d\Pi_t(Q,P,Z;\overline{Q},\overline{P},\overline{Z})
\\&\leq \frac{C_B}{2}\int_{\R^{12 d}}\left[|(P-p)-(\overline{P}-\overline{p})|^2+|(Q-q)-(\overline{Q}-\overline{q})|^2\right]\,d\Pi_t(q,p,z;\overline{q},\overline{p},\overline{z})\, d\Pi_t(Q,P,Z;\overline{Q},\overline{P},\overline{Z})
\\&=\frac{C_B}{2}\int_{\R^{12 d}}\left[|(P-\overline{P})-(p-\overline{p})|^2+|(Q-\overline{Q})-(q-\overline{q})|^2\right]\,d\Pi_t(q,p,z;\overline{q},\overline{p},\overline{z})\, d\Pi_t(Q,P,Z;\overline{Q},\overline{P},\overline{Z})
\\&\leq C_B \E (|q_t|^2+|p_t|^2),
\end{align*}
i.e.,
\begin{equation}
\label{eq: estimate 4}
-2 \E \left[p_t\cdot (B\ast\rho_t(Q_t)-B\ast_q f[\overline{\rho}_t](\overline{Q}_t))\right]\leq C_B \E (|q_t|^2+|p_t|^2).
\end{equation}
In the same way, we also obtain
\begin{equation}
\label{eq: estimate 5}
-2 \E \left[z_t\cdot (B\ast\rho_t(Q_t)-B\ast\overline{\rho}_t(\overline{Q}_t))\right]\leq C_B \E (|q_t|^2+|z_t|^2).
\end{equation}
Substituting estimates from \eqref{eq: estimate 2} to \eqref{eq: estimate 5} into \eqref{eq: estimate 1}, we get
\begin{align}
\label{eq: estimate 6}
\frac{d}{dt}\Qq(q_t,p_t,z_t)&\leq (2a_1-2\lambda a_5-2a_2\beta)\,q_t\cdot p_t+[2\lambda a_4-2\alpha a_5-2\beta]\,q_t\cdot z_t\nonumber
\\&\qquad+[2\lambda a_2-2\lambda a_3+2a_5-2\alpha]\,p_t\cdot z_t -[2\beta a_4-(a_2+2a_4+1)(C_A+C_B)]\, |q_t|^2\nonumber
\\&\qquad-[2\lambda-2a_4-a_2 (C_A+C_B))\, |p_t|^2-[2\alpha a_3-2\lambda- (C_A+C_B)]\,|z_t|^2.
\end{align}
Set $\eta :=C_A+C_B$. We choose $a_1,\ldots, a_5$ such that
\begin{equation}
\label{eq: a1-a5}
\begin{cases}
2a_1-2\lambda a_5-2a_2\beta=0,
\\2\lambda a_4-2\alpha a_5-2\beta=0,
\\2\lambda a_2-2\lambda a_3+2a_5-2\alpha=0,
\\2\beta a_4 -(a_2+2a_4+1)\eta>0,\\
2\lambda-2a_4-a_2 \eta>0,\\
2\alpha a_3-2\lambda-\eta>0,
\end{cases}
\end{equation}
and that $\Qq(q,p,z)$ is a positive definite quadratic form on $\R^{3d}$. Firstly, we choose $a_4=\frac{\lambda}{2}$ and express $a_1,a_2,a_5$ in terms of $a_3$. Then we show that there exists $a_3$ and $\eta_0>0$ such that all the requirements are fulfilled for $0<\eta<\eta_0$.  Conditions in \eqref{eq: a1-a5} are equivalent to 
\begin{equation}
\label{eq: a1-a5 2}
\begin{cases}
a_5=\frac{\lambda a_4-\beta}{\alpha}=\frac{\lambda^2/2-\beta}{\alpha},\\
a_2=a_3+\frac{\alpha}{\lambda}-\frac{a_5}{\lambda},\\	
a_1=\lambda a_5+\beta a_2=\lambda a_5+\beta \Big(a_3+\frac{\alpha}{\lambda}-\frac{a_5}{\lambda}\Big),\\
2\beta a_4-(a_2+2a_4+1)\,\eta=\lambda\beta -(a_3+\frac{\alpha}{\lambda}-\frac{a_5}{\lambda}+\lambda+1)\,\eta>0,\\
2\lambda-2a_4-a_2\eta=\lambda-\Big(a_3+\frac{\alpha}{\lambda}-\frac{a_5}{\lambda}\Big)\eta>0,\\
2 \alpha a_3-2\lambda-\eta>0.
\end{cases}
\end{equation}
%
Let $a_3=2+\tilde{a}_3$, where $\tilde{a}_3$ will be chosen later on. Then we have
\begin{align*}
\Qq(q,p,z)&= \Big(a_5(\lambda-\frac{\beta}{\lambda})+\beta(a_3+\frac{\alpha}{\lambda})\Big)\,|q|^2+ \Big(a_3+\frac{\alpha}{\lambda}-\frac{a_5}{\lambda}\Big)\, |p|^2+ a_3 |z|^2+\lambda\, q\cdot p
\\&\qquad-2a_5\, q\cdot z + 2 p\cdot z
\\&=\Big(a_5(\lambda-\frac{\beta}{\lambda})+\beta(\tilde{a}_3+\frac{\alpha}{\lambda}+2)\Big)\,|q|^2+ \Big(\tilde{a}_3+\frac{\alpha}{\lambda}+2-\frac{a_5}{\lambda}\Big)\, |p|^2+ (2+\tilde{a}_3 )|z|^2
\\&\qquad+\lambda\, q\cdot p-2a_5\, q\cdot z + 2 p\cdot z
\\&=\Big(a_5(\lambda-\frac{\beta}{\lambda})+\beta(\tilde{a}_3+\frac{\alpha}{\lambda}+2)-a_5^2-\frac{\lambda^2}{4}\Big)|q|^2+\Big(\tilde{a}_3+\frac{\alpha}{\lambda}-\frac{a_5}{\lambda}\Big)\, |p|^2+\tilde{a}_3|z|^2
\\&\qquad +\Big(\frac{\lambda}{2}q+p\Big)^2+(a_5q+z)^2+(p+z)^2.
\end{align*}
Now we choose $\tilde{a}_3$ such that
\begin{equation}
\label{eq: tilde a_3}
\begin{cases}
f_1(\tilde{a}_3):=a_5(\lambda-\frac{\beta}{\lambda})+\beta(\tilde{a}_3+\frac{\alpha}{\lambda}+2)-a_5^2-\frac{\lambda^2}{4}>0,\\
f_2(\tilde{a}_3):=\tilde{a}_3+\frac{\alpha}{\lambda}-\frac{a_5}{\lambda}>0,\\
f_3(\tilde{a}_3):=\tilde{a}_3>0,\\
f_4(\tilde{a}_3):=\tilde{a}_3+\frac{\alpha}{\lambda}+\lambda+3-\frac{a_5}{\lambda}>0,\\
f_5(\tilde{a}_3):=2\alpha(2+\tilde{a}_3)-2\lambda>0.
\end{cases}
\end{equation}
Since all $f_i$, with $i=1,\ldots, 5$, are  linear functions of $\tilde{a}_3$ with positive slopes, it is obvious that there exists $\tilde{a}_3$ that is large enough so that all of them are positive. We choose such a $\tilde{a}_3$ and define
\begin{equation}
\label{eq: eta_0}
\eta_0:=\min\Big\{\frac{\lambda\beta}{f_4(\tilde{a}_3)},\frac{\lambda}{2+f_2(\tilde{a_3})},f_5(\tilde{a}_3)\Big\}>0.
\end{equation}
It then follows from \eqref{eq: tilde a_3} and \eqref{eq: eta_0} that \eqref{eq: a1-a5 2} are fulfilled for every $0<\eta<\eta_0$ and $\Qq$ is a positive definite quadratic on $\R^{3d}$.

From \eqref{eq: estimate 6} it follows that there exists a positive constant $C$, depending only on $\alpha,\beta,\lambda, C_A$ and $C_B$ such that
\begin{equation}
\label{eq: estimate 8}
\frac{d}{dt}\E\,\Qq(q_t,p_t,z_t)\leq -C\,\E\,[|q_t|^2+|p_t|^2+|z_t|^2], \quad\text{for all}\quad t\geq 0.
\end{equation}
Since in a finite dimensional space, all norms are equivalent to the Euclidean norm, the right hand side of \eqref{eq: estimate 8} is bounded above by $-C\,\E\,\Qq(q_t,p_t,z_t)$ for some positive constant $C$, i.e.,
\begin{equation}
\frac{d}{dt}\E\,\Qq(q_t,p_t,z_t)\leq -C\,\E\,\Qq(q_t,p_t,z_t), \quad\text{for all}\quad t\geq 0.
\end{equation}
By Gronwall's inequality, we obtain that
\begin{equation*}
\E\,\Qq(q_t,p_t,z_t)\leq e^{-Ct}\,\E\,\Qq(q_0,p_0,z_0),\quad\text{for all}\quad t\geq 0.
\end{equation*}
We can re-write the above inequality using definition of $(q_t,p_t,z_t)$ as follows
\begin{equation*}
\E\,\Qq((Q_t,P_t,Z_t)-(\overline{Q}_t,\overline{P}_t,\overline{Z}_t))\leq e^{-Ct}\,\E\,\Qq((Q_0,P_0,Z_0)-(\overline{Q}_0,\overline{P}_0,\overline{Z}_0)),
\end{equation*}
for all $t\geq 0$. Now we optimize over $(Q_0,P_0,Z_0)$ and $(\overline{Q}_0,\overline{P}_0,\overline{Z}_0)$ with respective laws $\rho_0$ and $\overline{\rho}_0$ to get
\begin{equation}
\E\Qq((Q_0,P_0,Z_0)-(\overline{Q}_0,\overline{P}_0,\overline{Z}_0))=W_{2\Qq}(\rho_0,\overline{\rho}_0)^2.
\end{equation}
Then using the relation $W_{2\Qq}(\rho_t,\overline{\rho}_t)^2\leq \E \Qq((Q_t,P_t,Z_t)-(\overline{Q}_t,\overline{P}_t,\overline{Z}_t))$, we have
\begin{equation}
W_{2\Qq}(\rho_t,\overline{\rho}_t)\leq e^{-C\,t}\,W_{2\Qq}(\rho_0,\overline{\rho}_0).
\end{equation}
This completes the proof of the proposition. 
\end{proof}
The following lemma will be used later to prove the existence and uniqueness of the stationary measure.
\begin{lem}
\label{lem: stationary}
\cite[Lemma 7.3]{CT07} Let $(\M,\mathrm{dist})$ be a complete metric space and $S(t)$ be a continuous semigroup on $(\M,\mathrm{dist})$. Assume that there exists $0<L(t)<1$ such that
\begin{equation*}
\mathrm{dist}(S(t)(x),S(t)(y))\leq L(t)\mathrm{dist}(x,y),
\end{equation*}
for all $t>0$, and $x,y$ in $\M$. Then there exists a unique stationary point $x_\infty\in\M$, i.e., $S(t)(x_\infty)=x_\infty$ for all $t\geq 0$.
\end{lem}
We are now ready to present the proof of Theorem \ref{theo: longtime}.
\begin{proof}[Proof of Theorem \ref{theo: longtime}] Let $\Qq(q,p,z)$ be the positive definite quadratic form on $\R^{3d}$ obtained in Proposition \ref{pro: quadratic}. Since $\Qq$ is equivalent to the canonical form  $|q|^2+|p|^2+|z|^2$, there exists positive constants $C'$ and $C''$ such that
\begin{equation*}
W_{2}(\rho_t,\overline{\rho}_t)\leq C'' \, W_{2\Qq}(\rho_t,\overline{\rho}_t)\leq C'' e^{-C\,t}\,W_{2\Qq}(\rho_0,\overline{\rho}_0)\leq C'  e^{-C\,t}\,W_{2}(\rho_0,\overline{\rho}_0), 
\end{equation*}
for all $t\geq 0$ and all solutions $(\rho_t)_{t\geq 0}$ and $(\overline{\rho}_t)_{t\geq 0}$ of \eqref{eq: GLE} by Proposition \ref{pro: quadratic}. This  proves the first assertion \eqref{eq: assertion 1 thm1} of Theorem \ref{theo: longtime}. Since $\sqrt{\Qq}$ is a norm on $\R^{3d}$, which is equivalent to the Euclidean norm $\sqrt{|q|^2+|p|^2+|z|^2}$, according to \cite{Bolley08}, the space $(\mathcal{P}_2(\R^{3d}),W_{2\Qq})$ is a complete metric space. The existence of a unique of stationary measure $\rho_{\infty}$ is then followed from the contraction property of Proposition \ref{pro: quadratic} and Lemma \ref{lem: stationary}. Moreover, taking $(\overline{\rho}_t)_{t\geq 0}\equiv \rho_\infty$  in \eqref{eq: assertion 1 thm1} we obtain \eqref{eq: assertion 2 thm1}.
\end{proof}
\subsection{Particle approximation of the main equation: proof of Theorem \ref{theo: progation of chaos}}
\label{sec: propa}
We start with proving that the second moments of all solutions of \eqref{eq: GLE} are finite provided that the second moment of the initial data is finite.
\begin{lem}
\label{lem: bounded of second moments}
Under Assumption \ref{ass: assumption 1}, for all positive $\alpha,\beta$ and $\lambda$ there exists a positive constant $\eta_0$ such that, if $0\leq C_A+C_B<\eta_0$, then $\sup_{t\geq 0}\int_{\R^{3d}}(|q|^2+|p|^2+|z|^2) d \rho_t(q,p,z)$ is finite for all solutions $(\rho_t)_{t\geq 0}$ of \eqref{eq: GLE} with initial data $\rho_0\in \mathcal{P}_2(\R^{3d})$.
\end{lem}
The assertion of this lemma can be obtained using Theorem \ref{theo: longtime}. Indeed, let $\rho_\infty$ be the unique stationary measures of \eqref{eq: GLE} obtained in Theorem \ref{theo: longtime}. For all solutions $(\rho_t)_{t\geq 0}$ of \eqref{eq: GLE} with initial data $\rho_0\in \mathcal{P}_2(\R^{3d})$, from Theorem \ref{theo: longtime}, we have for all $t\geq 0$
\begin{align*}
\int_{\R^{3d}}(|q|^2+|p|^2+|z|^2) d \rho_t(q,p,z)&=W_2(\rho_t,\delta_{0})^2
\\&\leq 2\left(W_2(\rho_t,\rho_\infty)^2+W_2(\delta_0,\rho_\infty)^2\right)
\\&\overset{\eqref{eq: assertion 2 thm1}}{\leq}  2\left( C'\,e^{-C\, t}\,W_2(\rho_0,\rho_\infty)^2+W_2(\delta_0,\rho_\infty)^2\right)
\\& \leq 2\left( C'\,W_2(\rho_0,\rho_\infty)^2+W_2(\delta_0,\rho_\infty)^2\right).
\end{align*}
However, to show that the finiteness of the second moments, and hence the property of propagation of chaos of the system \eqref{eq: many-SDE}, is independent of the long time behaviour of the solutions, we provide a direct proof using again the coupling technique similarly as in Theorem \ref{theo: longtime}. This proof also demonstrates further the usefulness of the coupling technique.
\begin{proof}[A direct proof of Lemma \ref{lem: bounded of second moments} using coupling technique] Let $\rho_t$ be a solution of \eqref{eq: GLE} with initial data $\rho_0$ in $\mathcal{P}_2(\R^{3d})$. Let $a_1,\cdots,a_5$ be positive constants, which will be specified later on. Define
\begin{equation}
\label{eq: Q1 second moment}
\Qq (q,p,z)=a_1|q|^2+a_2|p|^2+a_3|z|^2+2a_4 q\cdot p+2a_5 q\cdot z+2 p\cdot z.
\end{equation}
Then
\begin{align}
&\frac{d}{dt}\int_{\R^{3d}} \Qq(q,p,z)\rho_t(dqdpdz)=\int_{\R^{3d}} \Qq(q,p,z)\partial_t\rho_t\,dqdpdz\nonumber
\\&\qquad=\int_{\R^{3d}}\Qq(q,p,z)\Big(-\div_q(p\rho_t)+\div_p[(\beta q+A(q)+B\ast \rho_t-\lambda z)\rho_t]\nonumber
\\&\hspace*{4.5cm}+\div_z[(\lambda p+\alpha z)\rho_t]+\Delta_z \rho_t\Big)\,dqdpdz\nonumber
\\&\qquad=2a_3 d + 2\int_{\R^{3d}}\Big[p\cdot(a_1 q+a_4p+a_5z)-(\beta q+A(q)+B\ast\rho_t(q)-\lambda z)\cdot (a_2 p+a_4q+z)\nonumber
\\&\hspace*{4cm}-(\lambda p+\alpha z)\cdot(a_3z +a_5 q+p)\Big]\rho_t(dqdpdz)\nonumber
\\&\qquad=2a_3d+2\int_{\R^{3d}}\Big[-\beta a_4\,|q|^2+(a_1-\beta a_2-\lambda a_5)\,q\cdot p+(-\beta +\lambda a_4-\alpha a_5)\,q\cdot z\nonumber
\\&\hspace*{4cm}-(\lambda-a_4)\,|p|^2+(a_5+\lambda a_2-\lambda a_3-\alpha)\,p\cdot z-(\alpha a_3-\lambda)\,|z|^2\nonumber
\\&\hspace*{5cm}-(A(q)+B\ast \rho_t(q))\cdot(a_2p+a_4q+z)\Big]\rho_t(dqdpdz).\label{eq: estimate Q 1}
\end{align}
We now estimate the last two terms involving $A$ and $B$ in \eqref{eq: estimate Q 1} using Young's inequality and assumptions on $A$ and $B$.

For the $A$-terms, we get
\begin{align}
-2 (a_2 p+a_4 q+z)\cdot A(q)&=-2(a_2 p+a_4 q+z)\cdot(A(q)-A(0)+A(0))\nonumber
\\&\leq 2C_A(a_2 |p|+a_4|q|+|z|)|q|-2(a_2 p+a_4 q+z)\cdot A(0),\nonumber
\\&\leq C_A((2a_4+a_2+1)|q|^2+a_2|p|^2+|z|^2)+(\beta a_4 |q|^2+\lambda |p|^2+\alpha a_3 |z|^2)\nonumber
\\&\qquad+\left(\frac{a_4}{\beta}+\frac{a_2^2}{\lambda}+\frac{1}{\alpha a_3}\right)|A(0)|^2.\label{A terms}
\end{align}
For the $B$-terms, we obtain
\begin{align}
&-2\int_{\R^{3d}}(a_2 p+a_4 q+z)\cdot B\ast \rho_t(q)\,\rho_t(dqdpdz)\nonumber
\\&\qquad=-2\int_{\R^{6d}}(a_2 p+a_4 q+z)\cdot B(q-q')\rho_t(dq'dp'dz') \rho_t(dq,dp,dz)\nonumber
\\&\qquad =-\int_{\R^{6d}}[a_2(p-p')+a_4(q-q')+(z-z')]\cdot B(q-q')\rho_t(dq'dp'dz') \rho_t(dq,dp,dz)\nonumber
\\&\qquad\leq C_B\,\int_{\R^{6d}}[a_4|q-q'|+a_2|p-p'|+|z-z'|]\,|q-q'|\rho_t(dq'dp'dz') \rho_t(dq,dp,dz)\nonumber
\\&\qquad\leq C_B\,\int_{\R^{6d}}\Big[a_4|q-q'|^2+\frac{a_2}{2}(|q-q'|^2+|p-p'|^2)\nonumber
\\&\hspace*{5cm}+\frac{1}{2}(|q-q'|^2+|z-z'|^2)\Big]\,\rho_t(dq'dp'dz') \rho_t(dq,dp,dz)\nonumber
\\&\qquad\overset{(*)}{\leq} C_B\,\int_{\R^{3d}}\Big[2a_4|q|^2+a_2(|q|^2+|p|^2)+(|q|^2+|z|^2)\Big]\,\rho_t(dq,dp,dz)\nonumber
\\&\qquad=C_B\int_{\R^{3d}}\big[(2a_4+a_2+1)|q|^2+a_2|p^2|+|z|^2\big]\,\rho_t(dq,dp,dz),\label{B terms}
\end{align}
where to obtain $(*)$ we have used the following estimate
\begin{align*}
\int_{\R^{6d}}|q-q'|^2\,\rho_t(dq'dp'dz') \rho_t(dq,dp,dz)&=\int_{\R^{6d}}(|q|^2+|q'|^2-2q\cdot q')\,\rho_t(dq'dp'dz') \rho_t(dq,dp,dz)
\\&=2\int_{\R^{3d}}|q|^2\rho_t(dq,dp,dz)-2\Big(\int_{\R^{3d}}q\rho_t(dq,dp,dz)\Big)^2
\\&\leq 2\int_{\R^{3d}}|q|^2\rho_t(dq,dp,dz),
\end{align*}
and similarly for other terms involving $|p-p'|^2$ and $|z-z'|^2$.

To proceed, we first choose $a_1,\ldots,a_5$ such that the terms involving inner products in \eqref{eq: estimate Q 1} vanish, i.e.,
\begin{equation}
\label{eq: a15 1}
\begin{cases}
a_1-\beta a_2-\lambda a_5=0,\\
-\beta +\lambda a_4-\alpha a_5=0,\\
a_5+\lambda a_2-\lambda a_3-\alpha =0.
\end{cases}
\end{equation}
Substituting \eqref{A terms} and \eqref{B terms} into \eqref{eq: estimate Q 1}, and setting $\eta:=C_A+C_B$, we get
\begin{align}
&\frac{d}{dt}\int_{\R^{3d}} \Qq(q,p,z)\rho_t(dqdpdz)\nonumber
\\&\leq 2a_3d+ \left(\frac{a_4}{\beta}+\frac{a_2^2}{\lambda}+\frac{1}{\alpha a_3}\right)|A(0)|^2\nonumber
\\&\qquad+\int_{\R^{3d}}\Big[-(\beta a_4-(2a_4+a_2+1)\eta))\,|q|^2-(\lambda-2a_4-a_2\eta)\,|p|^2\nonumber
\\&\hspace*{3cm}-(\alpha a_3-2\lambda-\eta)|z|^2\Big]\rho_t(dqdpdz).\label{Q estimate}
\end{align}
Similarly as in the proof of Theorem \ref{theo: longtime}, we can prove the existence of a positive constant $\eta_0$, depending only on the parameters of the equation, such that for all $0\leq C_A+C_B<\eta_0$, there exist $a_1,\ldots,a_5$ that satisfy \eqref{eq: a15 1} and 
\begin{equation}
\begin{cases}
\beta a_4-(2a_4+a_2+1)\eta>0,\\
\lambda-2a_4-a_2\eta>0,\\
\alpha a_3-2\lambda-\eta,
\end{cases}
\end{equation}
and that $\Qq(q,p,z)$ is a positive definite quadratic form on $\R^{3d}$. It follows from \eqref{Q estimate} that
\begin{align*}
\frac{d}{dt}\int \Qq(q,p,z)\rho_t(dqdpdz)&\leq C_1-C_2\int_{\R^{3d}}(|q|^2+|p|^2+|z|^2)\rho_t(dqdpdz)
\\&\leq C_1-C_3\int_{\R^{3d}}\Qq(q,p,z)\rho_t(dqdpdz),
\end{align*}
for some positive constants $C_1, C_2$ and $C_3$. Applying the Gronwall's lemma, we obtain that
\begin{equation*}
\sup_{t\geq 0}\int_{\R^{3d}}\Qq(q,p,z)\rho_t(dqdpdz)<\infty,
\end{equation*}
provided that initially $\int_{\R^{3d}}\Qq(q,p,z)\rho_0(dqdpdz)<\infty$. This is equivalent to the statement that  
\begin{equation*}
\sup_{t\geq 0}\int_{\R^{3d}}(|q|^2+|p|^2+|z|^2)\rho_t(dqdpdz)<\infty,
\end{equation*}
if initially $\rho_0$ has finite second moment. This completes the proof of the Lemma.
\end{proof}
We are now in the position to prove Theorem \ref{theo: progation of chaos}.
\begin{proof}[Proof of Theorem \ref{theo: progation of chaos}]
We recall that for each $1\leq i\leq N$, the law $\mu_t$ of $(\Q^{i}_t,\P^{i}_t,\Z^{i}_t)$ is the solution to Eq. \eqref{eq: GLE} at time $t$ with the initial datum $\rho_0$ and that the two processes $(\Q^{i}_t,\P^{i}_t,\Z^{i}_t)_{t\geq 0}$ and $(Q^{i,N}_t,P^{i,N}_t,Z^{i,N}_t)_{t\geq 0}$ are driven by the same Brownian motions and initial datum. Define $(q^i_t,p^i_t,z^i_t):=(Q^{i,N}_t,P^{i,N}_t,Z^{i,N}_t)-(\Q^{i}_t,\P^{i}_t,\Z^{i}_t)$, then $(q^i_t,p^i_t,z^i_t)$ satisfies the following SDE
\begin{equation}
\label{eq: SDE difference}
\begin{cases}
dq^i_t=p^i_t\,dt,\\
dp^i_t=-\beta q^i_t\,dt- (A(Q^{i,N}_t)-A(\Q^{i}_t))\,dt-\frac{1}{N}\sum_{j=1}^N(B(Q^{i,N}_t-Q^{j,N}_t)-B\ast\mu_t(\Q^i_t))\,dt+\lambda z^i_t\,dt,\\
dz^i_t=-\lambda p^i_t\,dt-\lambda \alpha z^i_t\,dt,\\
(q^i_0,p^i_0,z^i_0)=(0,0,0).
\end{cases}
\end{equation}
The underlying idea of remaining of the proof will be similar as that of Theorem \ref{theo: longtime} and Lemma \ref{lem: bounded of second moments}. We will find a positive definite quadratic form $\D$ on $\R^{3d}$ such that the SDE \eqref{eq: SDE difference} is dissipative with respect to $\D$. Let $a_1,\ldots,a_5$ be positive constants to be chosen later on. We define
\begin{equation}
\D(q,p,z):=a_1|q|^2+a_2|p|^2+a_3|z|^2+2a_4 q\cdot p+2a_5 q\cdot z+2 p\cdot z.
\end{equation}
Let $(q^i_t,p^i_t,z^i_t)$ be a solution to the SDE \eqref{eq: SDE difference}. We have
\begin{align}
\label{eq: D 1}
&\frac{d}{dt}\D(q^i_t,p^i_t,z^i_t)\nonumber
\\&\quad=\frac{d}{dt}(a_1|q^i_t|^2+a_2|p^i_t|^2+a_3|z^i_t|^2+2a_4 q^i_t\cdot p^i_t+2a_5 q^i_t\cdot z^i_t+2 p^i_t\cdot z^i_t)\nonumber
\\&\quad= 2 a_1 q^i_t\cdot p^i_t+ 2 a_4 |p^i_t|^2+2 a_5 p^i_t\cdot z^i_t\nonumber
\\&\quad\quad+(2a_3 z^i_t+ 2a_5 q^i_t+2 p^i_t)\cdot\left(-\lambda p^i_t-\alpha z^i_t\right)\nonumber
\\&\quad\quad+ (2 a_2 p^i_t+2 a_4 q^i_t+ 2z^i_t)\cdot \Big(-\beta q^i_t- (A(Q^{i,N}_t)-A(\Q^{i}_t))-\frac{1}{N}\sum_{j=1}^N(B(Q^{i,N}_t-Q^{j,N}_t)\nonumber
\\&\hspace*{6cm}-B\ast\mu_t(\Q^i_t))+\lambda z^i_t\Big)\nonumber
\\&\quad=(2a_1-2\lambda a_5-2 \beta a_2)\,q^i_t\cdot p^i_t+
(2a_5-2\lambda a_3-2\alpha+2\lambda a_2)\,p^i_t\cdot z^i_t+(2\lambda a_4-2\alpha a_5-2\beta)\,q^i_t\cdot z^i_t\nonumber
\\&\quad\quad -2\beta a_4\,|q^i_t|^2-(2\lambda-2a_4)\,|p^i_t|^2-(2\alpha a_3-2\lambda)\,|z^i_t|^2\nonumber
\\&\quad\quad -(2 a_2 p^i_t+2 a_4 q^i_t+ 2z^i_t)\cdot(A(Q^{i,N}_t)-A(\Q^{i}_t))\nonumber
\\&\quad\quad -\frac{1}{N} \sum_{j=1}^N(2 a_2 p^i_t+2 a_4 q^i_t+ 2z^i_t)\cdot(B(Q^{i,N}_t-Q^{j,N}_t)-B\ast\mu_t(\Q^i_t)).
\end{align}
Similarly as in \eqref{eq: estimate 2} we have
\begin{align}
\label{eq: D 2}
 -(2 a_2 p^i_t+2a_4 q^i_t+2z^i_t)(A(Q^{i,N}_t)-A(\Q^i_t))&\leq 2 a_2 C_A|q^i_t||p^i_t|+2a_4 C_A|q^i_t|^2+2 C_A|q^i_t||z^i_t|\nonumber
 \\&\leq a_2 C_A (|q^i_t|^2+|p^i_t|^2)+2a_4 C_A|q^i_t|^2+C_A(|q^i_t|^2+|z^i_t|^2)\nonumber
 \\&= (a_2+2a_4+1)C_A |q^i_t|^2+ a_2 C_A |p^i_t|^2+C_A |z^i_t|^2.
\end{align}
Substituting \eqref{eq: D 2} into \eqref{eq: D 1}, we obtain
\begin{align}
\label{eq: D 3}
&\frac{d}{dt}\D(q^i_t,p^i_t,z^i_t)\nonumber
\\&\quad\leq (2a_1-2\lambda a_5-2 \beta a_2)\,q^i_t\cdot p^i_t+
(2a_5-2\lambda a_3-2\alpha+2\lambda a_2)\,p^i_t\cdot z^i_t+(2\lambda a_4-2 \alpha a_5-2\beta)\,q^i_t\cdot z^i_t\nonumber
\\&\quad\quad -(2\beta a_4-(a_2+2a_4+1)C_A)\,|q^i_t|^2-(2\lambda-2a_4-2a_2 C_A)\,|p^i_t|^2-(2\alpha a_3-2\lambda-C_A)\,|z^i_t|^2\nonumber
\\&\quad\quad -\frac{1}{N} \sum_{j=1}^N(2 a_2 p^i_t+2 a_4 q^i_t+ 2z^i_t)\cdot(B(Q^{i,N}_t-Q^{j,N}_t)-B\ast\mu_t(\Q^i_t)).
\end{align}
We first choose $a_1,\ldots, a_5$ such that the cross terms in the right-hand side of \eqref{eq: D 3} vanish, i.e.,
\begin{equation}
\label{eq: a1-a5 D}
\begin{cases}
2a_1-2\lambda a_5-2a_2\beta=0,
\\2\lambda a_2-2\lambda a_3+2a_5-2\alpha=0,
\\2\lambda a_4-2\alpha a_5-2\beta=0.
\end{cases}
\end{equation}
With this choice, \eqref{eq: D 3} reads
\begin{align}
\label{eq: D 5}
&\frac{d}{dt}\D(q^i_t,p^i_t,z^i_t)\nonumber
\\&\quad\leq -(2\beta a_4-(a_2+2a_4+1)C_A)\,|q^i_t|^2-(2\lambda-2a_4-2a_2 C_A)\,|p^i_t|^2-(2\alpha a_3-2\lambda-C_A)\,|z^i_t|^2\nonumber
\\&\quad\quad -\frac{1}{N} \sum_{j=1}^N(2 a_2 p^i_t+2 a_4 q^i_t+ 2z^i_t)\cdot(B(Q^{i,N}_t-Q^{j,N}_t)-B\ast\mu_t(\Q^i_t)).
\end{align}
Note that this estimate holds for every $i=1,\ldots, N$. By averaging \eqref{eq: D 5} over $i$ we have
\begin{align}
\label{eq: D 6}
&\frac{d}{dt}\D(q^1_t,p^1_t,z^1_t)\nonumber
\\&\quad\leq -(2\beta a_4-(a_2+2a_4+1)C_A)\,|q^1_t|^2-(2\lambda-2a_4-2a_2 C_A)\,|p^1_t|^2-(2\alpha a_3-2\lambda-C_A)\,|z^1_t|^2\nonumber
\\&\quad\quad -\frac{2}{N^2} \sum_{i,j=1}^N(a_2 p^i_t+a_4 q^i_t+ z^i_t)\cdot(B(Q^{i,N}_t-Q^{j,N}_t)-B\ast\mu_t(\Q^i_t)).
\end{align}
Next we estimate the last term in \eqref{eq: D 6}. Using the following identity
\begin{equation*}
B(Q^{i,N}_t-Q^{j,N}_t)-B\ast\mu_t(\Q^i_t)=B(Q^{i,N}_t-Q^{j,N}_t)-B(\Q^{i}_t-\Q^{j}_t)+B(\Q^{i}_t-\Q^{j}_t)-B\ast\mu_t(\Q^i_t),
\end{equation*}
we now decompose the last term in \eqref{eq: D 6} into six terms  and estimate them as follows.
\begin{enumerate}[1)]
\item By symmetry and the Lipschitz property of $B$, we have
\begin{align*}
\label{eq: term 1}
&-\sum_{i,j=1}^N\E [q^i_t\cdot (B(Q^{i,N}_t-Q^{j,N}_t)-B(\Q^{i}_t-\Q^{j}_t))]\nonumber
\\&\quad=-\sum_{i,j=1}^N\E [(Q^{i,N}_t-\Q^{i}_t)\cdot (B(Q^{i,N}_t-Q^{j,N}_t)-B(\Q^{i}_t-\Q^{j}_t))]\nonumber
\\&\quad=-\frac{1}{2}\sum_{i,j=1}^N\E [((Q^{i,N}_t-Q^{j,N}_t)-(\Q^{i}_t-\Q^{j}_t))\cdot (B(Q^{i,N}_t-Q^{j,N}_t)-B(\Q^{i}_t-\Q^{j}_t))]\nonumber
\\&\quad\leq \frac{C_B}{2}\sum_{i,j=1}^N\E \left|(Q^{i,N}_t-Q^{j,N}_t)-(\Q^{i}_t-\Q^{j}_t)\right|^2\nonumber
\\&\quad =\frac{C_B}{2}\sum_{i,j=1}^N\E \left|(Q^{i,N}_t-\Q^{i}_t)-(Q^{j,N}_t-\Q^{j}_t)\right|^2\nonumber
\\&\quad =C_B\sum_{i,j=1}^N\E |q^i_t|^2-C_B\E\,\Big|\sum_{i=1}^N q^i_t\Big|^2\nonumber
\\&\quad\leq C_B\sum_{i,j=1}^N\E |q^i_t|^2=C_B\,N^2\,\E |q^1_t|^2.
\end{align*}
\item By assumption on $B$ and Young's inequality, we have
\begin{align*}
&-\sum_{i,j=1}^N\E [p^i_t\cdot (B(Q^{i,N}_t-Q^{j,N}_t)-B(\Q^{i}_t-\Q^{j}_t))]\nonumber
\\&\quad=-\frac{1}{2}\sum_{i,j=1}^N\E [(p^{i}_t-p^{j}_t)\cdot (B(Q^{i,N}_t-Q^{j,N}_t)-B(\Q^{i}_t-\Q^{j}_t))]\nonumber
\\&\quad\leq \frac{C_B}{2}\sum_{i,j=1}^N\E \Big[|p^{i}_t-p^{j}_t|\big|(Q^{i,N}_t-Q^{j,N}_t)-(\Q^{i}_t-\Q^{j}_t)\big|\Big]\nonumber
\\&\quad\leq \frac{C_B}{2}\sum_{i,j=1}^N\E \Big[\frac{1}{2}|p^{i}_t-p^{j}_t|^2+\frac{1}{2}\big|(Q^{i,N}_t-Q^{j,N}_t)-(\Q^{i}_t-\Q^{j}_t)\big|^2\Big]\nonumber
\\&\quad=\frac{C_B}{2}\left(N^2 \E |p^1_t|^2-\big|\sum_{j=1}^N\E p^j_t\big|^2+N^2 \E|q^1_t|^2-\big|\sum_{j=1}^N\E q^j_t\big|^2\right)
\\&\quad\leq \frac{C_B}{2}N^2\,\E [|q^1_t|^2+|p^1_t|^2].
\end{align*}
\item Similarly we obtain
\begin{equation*}
\label{eq: term 3}
-\sum_{i,j=1}^N\E [z^i_t\cdot (B(Q^{i,N}_t-Q^{j,N}_t)-B(\Q^{i}_t-\Q^{j}_t))]\leq \frac{C_B}{2}\,N^2\,\E [|q^1_t|^2+|z^1_t|^2].
\end{equation*}
\item We continue with the term
\begin{align}
\label{eq: term 4}
&-2a_4\E\,\Big[q^i_t\cdot\sum_{j=1}^N \big(B(\Q^{i}_t-\Q^{j}_t)-B\ast\mu_t(\Q^i_t)\big)\Big]\nonumber
\\&\quad\leq N\beta a_4\,\E\,|q^i_t|^2+\frac{a_4}{N\beta }\,\E\Big|\sum_{j=1}^N\big(B(\Q^{i}_t-\Q^{j}_t)-B\ast\mu_t(\Q^i_t)\big)\Big|^2\nonumber
\\&\quad = N\beta a_4\,\E\,|q^i_t|^2+\frac{a_4}{N\beta}\,\sum_{j=1}^N\E\Big|B(\Q^{i}_t-\Q^{j}_t)-B\ast\mu_t(\Q^i_t)\Big|^2\nonumber
\\&\qquad\quad+\frac{a_4}{N\beta}\sum_{j\neq k}\E\Big[(B(\Q^{i}_t-\Q^{j}_t)-B\ast \mu_t(\Q^i_t))\cdot(B(\Q^{i}_t-\Q^{k}_t)-B\ast\mu_t(\Q^i_t)) \Big].
\end{align}
We need to analyse further the last two terms in \eqref{eq: term 4}. Because $B$ is odd, we have $B(0)=0$. Hence for any $y \in \R^{d}$ it holds that $|B(y)|=|B(y)-B(0)|\leq C_B |y|$. It implies that
\begin{align}
\label{eq: term 4.1}
&\E\Big|B(\Q^{i}_t-\Q^{j}_t)-B\ast\mu_t(\Q^i_t)\Big|^2\nonumber
\\\quad &\leq 2\left[\E\,|B(\Q^{i}_t-\Q^{j}_t)|^2+\E|B\ast\mu_t(\Q^i_t)|^2\right]\nonumber
\\&\leq 2 C_B^2\left[\E\,|\Q^{i}_t-\Q^{j}_t|^2+\int_{\R^{6d}}|q-q'|^2\mu_t(q,p,z)\mu_t(q',p',z')\, dqdpdzdq'dp'dz'\right]\nonumber
\\\quad & \leq 8 C_B^2\int_{\R^{3d}} |q|^2\mu_t(q,p,z)\,dqdpdz\nonumber
\\\quad & \leq M,
\end{align}
where in the last inequality above we have used Lemma \ref{lem: bounded of second moments}  to obtain \\$M=16 C_B^2 \left( C'\,W_2(\rho_0,\rho_\infty)^2+W_2(\delta_0,\rho_\infty)^2\right)$. Certainly $M$ depends only on the initial second moment and the coefficients of the equation and but not on $t$ or $N$.

The last term in \eqref{eq: term 4} vanishes since for all $j\neq k$, we have
\begin{align}
\label{eq: term 4.2}
&\E\Big[(B(\Q^{i}_t-\Q^{j}_t)-B\ast\mu_t(\Q^i_t))\cdot(B(\Q^{i}_t-\Q^{k}_t)-B\ast\mu_t(\Q^i_t)) \Big]\nonumber
\\\qquad& =\E_{\Q^i_t}\Big[\big(\E_{\Q^j_t}[B(\Q^{i}_t-\Q^{j}_t)-B\ast\mu_t(\Q^i_t)]\big)\cdot\big(\E_{\Q^k_t}[B(\Q^{i}_t-\Q^{k}_t)-B\ast\mu_t(\Q^i_t)]\big) \Big]\nonumber
\\\qquad& = \E_{\Q^i_t}[0]=0.
\end{align}
The equality above holds true because $\Q^j_t$ and $\Q^k_t$ are independent and have the same law, which is the first marginal of $\mu_t$.

Substituting \eqref{eq: term 4.1} and \eqref{eq: term 4.2} into \eqref{eq: term 4}, we get
\begin{equation*}
-2a_4\E\,\Big[q^i_t\cdot\sum_{j=1}^N \big(B(\Q^{i}_t-\Q^{j}_t)-B\ast\mu_t(\Q^i_t)\big)\Big]\leq N\beta a_4\E|q^i_t|^2+\frac{a_4}{\beta}\,M,
\end{equation*}
from which, by summing over $i$, we obtain
\begin{align}
\label{eq: term 4 final}
-2a_4\sum_{i,j=1}^N\E\,\Big[q^i_t\cdot\big(B(\Q^{i}_t-\Q^{j}_t)-B\ast\mu_t(\Q^i_t)\big)\Big]&\leq N^2\beta a_4\E|q^i_t|^2+\frac{a_4}{\beta}\,M N\nonumber
\\&=N^2\beta a_4\,\E|q^1_t|^2+\frac{a_4}{\beta}\,M N.
\end{align}
\item Similarly we obtain the following estimate
\begin{equation*}
\label{eq: term 5}
-2a_2\sum_{i,j=1}^N\E\,\Big[p^i_t\cdot\big(B(\Q^{i}_t-\Q^{j}_t)-B\ast\mu_t(\Q^i_t)\big)\Big]\leq N^2\lambda\,\E|p^1_t|^2+\frac{a_2^2}{\lambda}\, M N.
\end{equation*}
\item Finally, for the last term we also get
\begin{equation*}
\label{eq: term 6}
-2\sum_{i,j=1}^N\E\,\Big[z^i_t\cdot\big(B(\Q^{i}_t-\Q^{j}_t)-B\ast\mu_t(\Q^i_t)\big)\Big]\leq N^2\alpha a_3\,\E|p^1_t|^2+\frac{1}{\alpha a_3} M N.
\end{equation*}
\end{enumerate}
Set $\eta:=C_A+C_B$. Bringing all terms together, it follows from \eqref{eq: D 6} that
\begin{align}
\label{eq: D 8}
\frac{d}{dt}\D(q^1_t,p^1_t,z^1_t)&\leq -(\beta a_4-(a_2+2a_4+1)\eta)\,|q^1_t|^2-(\lambda-2a_4-2a_2 \eta)\,|p^1_t|^2\nonumber
\\&\qquad -(\alpha a_3-2\lambda-\eta)\,|z^1_t|^2+\Big(\frac{a_4}{\beta}+\frac{a_2^2}{\lambda}+\frac{1}{\alpha a_3}\Big)\frac{M}{N}.
\end{align}
Note the difference between the right hand side of the above inequality and that of \eqref{eq: estimate 6} in Section \ref{sec: Long time}. The difference is due to the extra terms involving $B(\Q^i_t-\Q^j_t)-B\ast\mu_t(\Q^i_t)$ as shown in the above computations. We now choose $a_4=\frac{\lambda}{4}$. From \eqref{eq: a1-a5 D} we can express $a_1, a_2$ and $a_5$ in terms of $a_3$ and choose $a_3$ such that
\begin{equation}
\label{eq: a1-a5 D2}
\begin{cases}
a_5=\frac{\lambda a_4-\beta}{\alpha}=\frac{\lambda^2/4-\beta}{\alpha},\\
a_2=a_3+\frac{\alpha}{\lambda}-\frac{a_5}{\lambda},\\	
a_1=\lambda a_5+\beta a_2=\lambda a_5+\beta \Big(a_3+\frac{\alpha}{\lambda}-\frac{a_5}{\lambda}\Big),\\
\beta a_4-(a_2+2a_4+1)\,\eta=\frac{\lambda}{4}\beta -(a_3+\frac{\alpha}{\lambda}-\frac{a_5}{\lambda}+\lambda+1)\,\eta>0,\\
\lambda-2a_4-a_2\eta=\frac{\lambda}{2}-\Big(a_3+\frac{\alpha}{\lambda}-\frac{a_5}{\lambda}\Big)\eta>0,\\
\alpha a_3-2\lambda-\eta>0,
\end{cases}
\end{equation}
and that $\D(q,p,z)$ is a positive definite quadratic form on $\R^{3d}$, where
\begin{equation*}
\D(q,p,z)= \Big(a_5(\lambda-\frac{\beta}{\lambda})+\beta(a_3+\frac{\alpha}{\lambda})\Big)\,|q|^2+ \Big(a_3+\frac{\alpha}{\lambda}-\frac{a_5}{\lambda}\Big)\, |p|^2+ a_3 |z|^2+\frac{\lambda}{2}\, q\cdot p-2a_5\, q\cdot z + 2 p\cdot z.
\end{equation*}
Similarly as in the proof of Theorem \ref{pro: quadratic} and from \eqref{eq: D 8}, we obtain the existence of a positive constant $\eta_0$ depending only on $\alpha,\beta,\lambda, C_A$ and $C_B$ such that for all $0<C_A+C_B$, there exist $a_1,\ldots,a_5$ such that $\D(q,p,z)$ is a positive definite quadratic form on $\R^{3d}$ and such that
\begin{equation}
\label{eq: estimate D}
\frac{d}{dt}\E\D(q^1_t,p^1_t,z^1_t)\leq -C_1 \E[|q^1_t|^2+|p^1_t|^2+|z^1_t|^2]+\frac{C_2}{N},
\end{equation}
for all $t\geq 0$ and for positive constants $C_1$ and $C_2$, which depend only on the parameters of the equation and the initial second moment, but not on $N$. Since $\D$ is a positive quadratic on $\R^{3d}$, the right-hand side of \eqref{eq: estimate D} in turn is bounded by $-C_3 \E\D(q^1_t,p^1_t,z^1_t)+\frac{C_2}{N}$, for some positive constant $C_3$, so that
\begin{equation}
\frac{d}{dt}\E\D(q^1_t,p^1_t,z^1_t)\leq -C_3 \E\D(q^1_t,p^1_t,z^1_t)+\frac{C_2}{N}.
\end{equation}
From this inequality, we deduce that
\begin{equation*}
\E\D(q^1_t,p^1_t,z^1_t)\leq \frac{C_4}{N}, \quad\text{for all }~ t\geq 0,
\end{equation*} 
for some positive constant $C_4$. Using the fact that $\D$ is a positive definite quadratic on $\R^{3d}$ again, we obtain the following bound 
\begin{equation*}
\E[|Q^{1,N}_t-\Q^i_t|^2+|P^{1,N}_t-\P^i_t|^2+|Z^{1,N}_t-\Z^i_t|^2]=\E[|q^1_t|^2+|p^1_t|^2+|z^1_t|^2]\leq \frac{C}{N}, 
\end{equation*} 
for all $t\geq 0$ and for some positive constant $C$ depending only on the parameters of the equation and the initial second moment, but not on $N$. This completes the proof of \eqref{eq: propagation of chaos}. 

The assertion that the particle system \eqref{eq: many-SDE} satisfies the property of propagation of chaos is a direct consequence of \eqref{eq: propagation of chaos}. In deed, let $(i_1,\ldots,i_k)$ be an arbitrary $k$-uple of $[1,N]$ and $\varphi$ be a $1$-Lipschitz function on $(\R^{3d})^k$. Then
\begin{align*}
&\left|\int_{(\R^{3d})^k} \varphi\, \rho^{(i_1,\ldots i_k)}_t-\int_{(\R^{3d})^k} \varphi\,(\mu_t)^{\otimes k}\right|=\left|\E\varphi(X_t^{i_1,N},\ldots,X_t^{i_k,N})-\E\varphi(\X_t^{i_1,N},\ldots,\X_t^{i_k,N})\right|
\\& \qquad \leq \left|\E\varphi(X_t^{i_1,N},\ldots,X_t^{i_k,N})-\E\varphi(\X_t^{i_1,N},X_t^{i_2,N},\ldots,X_t^{i_k,N})\right|+\ldots
\\ & \qquad\qquad +\left|\E\varphi(\X_t^{i_1,N},\ldots,\X_t^{i_{k-1},N},X_t^{i_k,N})\right|-\E\left|\varphi(\X_t^{i_1,N},\ldots,\X_t^{i_k,N})\right|
\\&\qquad \leq \sum_{i=1}^k \E|X^{i,N}_t-\X^{i,N}_t|\leq \frac{C\,k}{\sqrt{N}},
\end{align*}
where we have used the Lipschitz property of $\varphi$ and \eqref{eq: propagation of chaos}. This estimate implies that $\rho^{i_1,\ldots,i_k}_t$ converges to $(\mu_t)^{\otimes k}$ in the $1$-Wasserstein distance, which is tested against $1$-Lipschitz functions. Since the $1$-Wasserstein distance is equivalent to narrow convergence plus convergence of the first moments, this concludes that $\rho^{i_1,\ldots,i_k}_t$ converges to $(\mu_t)^{\otimes k}$ narrowly. 

Finally, the estimate \eqref{eq: Wasserstein bound} is  followed straightforwardly thanks to \eqref{eq: propagation of chaos} and definition of the Wasserstein distance \eqref{eq: def W_2}. We complete the proof of Theorem \ref{theo: progation of chaos}.
\end{proof}





\end{document}